\documentclass{aims}

\usepackage{amsmath}
\usepackage{paralist}
\usepackage{graphics} 
\usepackage{epsfig}   
\usepackage{graphicx}  
\usepackage{epstopdf} 
\usepackage[colorlinks=true]{hyperref}

\hypersetup{urlcolor=blue, citecolor=red}

\def\currentvolume{X}
\def\currentissue{X}
\def\currentyear{200X}
\def\currentmonth{XX}
\def\ppages{X--XX}
\def\DOI{10.3934/xx.xx.xx.xx}
     
\newtheorem{theorem}{Theorem}[section]

\theoremstyle{definition}

\newtheorem{remark}{Remark}

\usepackage{graphicx}								
\usepackage{subfig}
\usepackage{amsfonts} 							
\usepackage{dsfont}
\usepackage{bigints} 								
\usepackage{algorithm}              
\usepackage{algorithmic}
\usepackage{booktabs} 							
\usepackage{multirow}								
\usepackage{setspace}
\usepackage{mathrsfs}

\usepackage{tikz}	
\usetikzlibrary{arrows,chains,matrix,positioning,scopes}

\makeatletter
\tikzset{join/.code=\tikzset{after node path={%
\ifx\tikzchainprevious\pgfutil@empty\else(\tikzchainprevious)%
edge[every join]#1(\tikzchaincurrent)\fi}}}

\makeatother
\tikzset{>=stealth',every on chain/.append style={join}, every join/.style={->}}
\tikzstyle{labeled}=[execute at begin node=$\scriptstyle,   execute at end node=$]

\newcommand {\Int}   {\int\limits}
\newcommand {\Sum}   {\sum\limits}

\newcommand {\IntO}  {\Int_\Omega}

\newcommand {\CF}    {C_{\mathrm{F\Omega}}}
\newcommand {\Ctr}   {C_{\mathrm{tr}}}

\newcommand {\CPi}   {C_{\mathrm{P\Omega}_i}}
\newcommand {\Rd}    {{\mathds{R}}^d}

\newcommand {\R}     {\mathscr{R}}
\newcommand {\I}     {\mathscr{I}}

\def \NormA#1  {{\mid\!\mid\!\mid\! #1 \!\mid\!\mid\!\mid}^2_{A} }   
\def \NormAinverse#1  { {\mid\!\mid\!\mid\! #1 \!\mid\!\mid\!\mid}^2_{A^{-1}} }   
\def \NormQT#1 {{\mid\!\mid\!\mid\! #1 \!\mid\!\mid\!\mid}^2_{Q_T}}   
\def \NormQO#1 {{\mid\!\mid\!\mid #1 \mid\!\mid\!\mid}^2_{Q^0}}   
\def \NormQk#1 {{\mid\!\mid\!\mid #1 \mid\!\mid\!\mid}^2_{Q^k}}   
\def \Normf#1  {\Big \lceil #1 \Big\rceil_\Omega}
\def \dvrg     {\mathrm{div} \:}	

\def \dt       {\mathrm{\:d}t}
\def \dx       {\mathrm{\:d}x}
\def \dxt      {\mathrm{\:d}x\mathrm{d}t}
\def \dst      {\mathrm{\:d}s\mathrm{d}t}

\def \Normt#1  {\mid\!\mid\!\mid #1 \mid\!\mid\!\mid}   
\def \Mean#1#2 {{ \Big \{ \big| #1 \big| \Big\} }_{#2}}
\def \smallMean#1#2 {{ \big \{ | #1 | \big\} }_{#2}}

\def\L#1{L^{#1}}
\def\H#1{H^{#1}}
\def\Ho#1{{\mathring{H}}^{#1}}

\def\majtwo{\overline{\mathrm M}^2_{\mathrm{II}}}
\def\majtwod{\overline{\mathrm M}^2_{\mathrm{II}, \mathrm{N}}}

\def\maj{{\overline{\mathrm M}^2_{\,\mathrm{I}}}}
\def\majmuzero{{\overline{\mathrm M}^2_{\,\mathrm{I},\, \mu = \mathrm{0}}}}
\def\majmuone{{\overline{\mathrm M}^2_{\,\mathrm{I},\, \mu = \mathrm{1}}}}
\def\majd{{\overline{\mathrm M}^2_{\,\mathrm{I}, \mathrm{N}}}}

\def\Min{{\underline{\mathrm M}^2}}
\def\error{{[e]\,}^2}

\title[]
{A posteriori error estimates for time-dependent reaction-diffusion
problems based on the Payne--Weinberger inequality}
	
\author[S. Matculevich and P. Neittaanm{\"a}ki and S. Repin]
{Svetlana Matculevich and Pekka Neittaanm{\"a}ki and Sergey Repin}	

\subjclass{Primary: 58F15, 58F17; Secondary: 53C35.}
\keywords{Parabolic equations, a posteriori estimates, Poincare type estimates.}

\email{svetlana.v.matculevich@jyu.fi}
\email{pekka.neittaanmaki@jyu.fi}
\email{repin@pdmi.ras.ru; serepin@jyu.fi}

\begin{document}


\centerline{\scshape Svetlana Matculevich}
\medskip
{\footnotesize
 \centerline{Dept. of Mathematical Information Technology, Faculty of Information Technology}
 \centerline{C321.4, Agora, P.O. Box 35, FI-40014, University of Jyv{\"a}skyl{\"a}, Finland}
} 
\medskip

\centerline{\scshape Pekka Neittaanm{\"a}ki}
\medskip
{\footnotesize
 \centerline{Dept. of Mathematical Information Technology, Faculty of Information Technology}
 \centerline{Agora, P.O. Box 35, FI-40014, University of Jyv{\"a}skyl{\"a}, Finland}
}
\medskip

\centerline{\scshape Sergey Repin}
{\footnotesize
 \centerline{Dept. of Mathematical Information Technology, Faculty of Information Technology}
 \centerline{Agora, P.O. Box 35, FI-40014, University of Jyv{\"a}skyl{\"a}, Finland}
 \vskip 5pt 
 \centerline{V.A. Steklov Institute of Mathematics at St. Petersburg}
 \centerline{191011, Fontanka 27, St.Petersburg, Russia}
}

\bigskip	


\begin{abstract}
We consider evolutionary reaction-diffusion problem with mixed Dirichlet--Robin 
boundary conditions. For this class of problems, we derive two-sided estimates of the 
distance between any function in the admissible energy space and exact solution of the 
problem. The estimates (majorants and minorants) are explicitly computable and do not
contain unknown functions or constants. Moreover, it is proved that the estimates are 
equivalent to the energy norm of the deviation from the exact solution. 	
\end{abstract}
\maketitle


\section{Problem statement}

Let $\Omega \in \Rd$ be a bounded connected domain with Lipchitz continuous boundary 
$\partial \Omega$, which consists of two measurable non-intersecting parts $\Gamma_D$ 
and $\Gamma_R$ associated with the Dirichlet and Robin boundary conditions, respectively. 
By $Q_T$ we denote the space-time cylinder $Q_T := \Omega \times (0, T)$, $T > 0$, and
$S_T := \partial\Omega \times [0, T] = \big(\Gamma_D \cup \Gamma_R \big) \times [0, T]$. 
The parts of $S_T$ related to $\Gamma_D$ and $\Gamma_R$ are denoted by $S_D$ and $S_R$, 
respectively. 


We consider the classical reaction-diffusion initial boundary value problem: find 
$u(x, t)$ and $p(x, t)$ such that 
\begin{alignat}{3}
   u_t - \nabla \cdot p + \lambda u & =\, f,	      & \quad (x, t) \in Q_T,\label{eq:parabolic-problem-equation}\\
                                  p & =\, A \nabla u, & \quad (x, t) \in Q_T,\nonumber\\
                      	    u(x, 0) & =\, \varphi,    & \quad x \in \Omega,\label{eq:parabolic-problem-initial-condition}\\
                        	  u & =\, 0,	      & \quad (x, t) \in S_D,\label{eq:parabolic-problem-dirichlet-boundary-condition}\\
   			       p \cdot n + \sigma u & =\, g,					& \quad (x, t) \in S_R,\label{eq:parabolic-problem-robin-boundary-condition}
\end{alignat}
where $n$ denotes the vector of unit outward normal to $\partial\Omega$, and
\begin{equation}
f(x, t) \in \L{2}(Q_T), \quad 
\varphi(x) \in \L{2}(\Omega), \quad
g(x, t) \in \L{2}\left(0, T; \L{2} (S_R)\right).
\label{eq:problem-condition}
\end{equation}

The function $\lambda$ entering the reaction part of 
(\ref{eq:parabolic-problem-equation}) is a non-negative bounded function, which values 
may vary from very small (or zero) to large values in different parts of the domain. 
The function $\sigma(s, t)$ is a bounded function defined on $\Gamma_R$. 
We assume that for any 
$(x, t) \in Q_T$ the matrix $A$ is symmetric and satisfies the condition
\begin{equation}
\nu_1 |\xi|^2 \leq A(x, t) \: \xi \cdot \xi \leq \nu_2 |\xi|^2,\quad
\xi \in \Rd,\quad 0 < \nu_1 \leq \nu_2 < \infty.
\label{eq:operator-a}
\end{equation}
%
%
By $\| \cdot \|_\Omega$ and $\| \cdot \|_{Q_T}$, we denote the standard norms in 
$\L{2}(\Omega)$ and 
$\L{2}(Q_T)$, respectively. $\L{2, 1} (Q_T)$ is the space of functions $g(x, t)$ 
with the finite norm 
$\Int_0^T \| g(\cdot, t)\|_{\Omega} \dt$, 
$\Ho{1}(Q_T)$ is a subspace of $\H{1}(Q_T)$, which contains functions satisfying 
(\ref{eq:parabolic-problem-dirichlet-boundary-condition}), \linebreak
$\H{1, 0} (Q_T) := \L{2}\left( 0, T; \H{1}(\Omega) \right)$, and
$V_2(Q_T) $  := 
$\H{1, 0}(Q_T) \cap \L{\infty}\left( 0, T; \L{2}(\Omega) \right)$.
%
%
The space $V^{1, 0}_2 (Q_T) := \H{1, 0} (Q_T) \cap C \left(0, T; \L{2}(\Omega)\right)$ 
is a subspace of $V_2(Q_T)$ with functions possessing $\L{2}$-traces defined for 
a.a. $t \in [0, T]$. 

The generalized solution of (\ref{eq:parabolic-problem-equation})--(\ref{eq:parabolic-problem-robin-boundary-condition}) 
is defined as a function $u(x, t) \in V^{1, 0}_2 (Q_T)$, satisfying 
the integral identity
\begin{multline}
	\Int_{\Omega} \Big( u(x, T) \eta(x, T) - u(x, 0) \eta(x, 0) \Big) \dx - 
	\Int_{Q_T} u \eta_t \dxt + 
        \Int_{Q_T} A \nabla{u} \cdot \nabla{\eta} \dxt \, + \\ 
	\Int_{S_R} \sigma u \eta \dst + \Int_{Q_T} \lambda u \eta \dxt = 
	\Int_{Q_T} f \eta \dxt + \Int_{S_R} g \eta \dst, \quad \forall \eta \in \Ho{1}(Q_T).
	\label{eq:generalized-statement}
\end{multline}
%
Classical solvability results (see, e.g., 
\cite{Ladyzhenskaya1985, Ladyzhenskayaetall1967, Evans2010}) guarantee that $u$
exists and is unique in $V^{1, 0}_2(Q_T)$.


Assume that $v \in \Ho{1}(Q_T)$ is an approximation of $u$. Our goal is to deduce explicitly 
computable and realistic estimates of the distance between $u$ and $v$. 
In other words, we wish to quantify neighborhoods of the exact solution in terms of local
topology equivalent to the natural energy norm. More precisely, we introduce the measure 
\begin{multline}	
	[ u - v ]^2_{(\nu, \theta, \zeta, \chi)} = 
	\nu\,    \NormA{\, \nabla (u - v)} \: + \: 
           \| \,\theta \, (u - v) \, \|^2_{Q_T} + \: \\
	\zeta \, \| \, (u - v)(\cdot, T) \, \|^2_{\Omega} + \: 
  \chi     \| \sqrt{\sigma}(u - v)\|^2_{S_R},
	\label{eq:energy-norm-for-reaction-diff-evolutionary-problem}
\end{multline}
where $\nu$, $\theta$, $\zeta$ and $\chi$ are certain positive weights (balancing 
different components of the error). They can be selected in different ways so that 
(\ref{eq:energy-norm-for-reaction-diff-evolutionary-problem}) presents a collection 
of different error measures. Here, 
\begin{equation}
\NormA{\tau} :=  \Int_{Q_T} A \tau \cdot \tau \dxt,
\end{equation}
henceforth, we also use the norms
\begin{equation*}
  \| \: \tau \: \|^2_A := \Int_\Omega A \tau \cdot \tau \dx, \quad 	
   \| \: \tau \: \|^2_{A^{-1}} := \Int_\Omega A^{-1} \tau \cdot \tau \dx, \quad
	\NormAinverse{\tau} := \Int_{Q_T} A^{-1} \tau \cdot \tau \dxt. 
\end{equation*}
%
In Theorem \ref{th:theorem-minimum-of-majorant-I}, we derive a fully computable and 
guaranteed upper bound of $e = u - v$ (for this purpose we use the method originally 
introduced in \cite{Repin2002}). In \cite{RepinTomar2010}, this method was applied to 
problems with convection, and in \cite{NeittaanmakiRepin2010} guaranteed error majorants 
were derived for the Stokes problem. In Section \ref{sec:majorant-I}, we combine this 
approach with the technique suggested in \cite{RepinSauter2006} for the stationary 
reaction-diffusion problem, which yields efficient bounds of the distance to the exact 
solution (error majorants) for problems with strongly changing reaction function. 

The majorant presented in Theorem 
\ref{th:theorem-minimum-of-majorant-I} contains the 
constant $\CF$ in the \linebreak Friedrichs type inequality (\ref{eq:friedrichs}). 
If $S_T = S_D$, then this constant (or a guaranteed upper bound of it) is easy to find. 
However, in the case of mixed boundary conditions and complicated domains, finding 
$\CF$ may cause a serious problem. Therefore, in
Theorems \ref{th:theorem-majorant-for-decomposed-domain-1} and 
\ref{th:theorem-majorant-for-decomposed-domain-2}, we derive another upper bounds, 
which are based on decomposition of $\Omega$ into a collection of non-overlapping convex 
sub-domains. By means of a technique close to that has been used in 
\cite{RepinDeGruyter2008} for elliptic problem, we deduce majorants, which involve only
constants in the Poincare type inequalities. For convex domains these constants are 
easy to estimate due to the well known result of Payne and Weinberger 
\cite{PayneWeinberger1960} (with correction of Bebendorf \cite{Bebendorf2003a}). 
Therefore, we obtain a fully computable error majorant (\ref{eq:majorant-1}), which 
involves only known data and constants. In Subsection 
\ref{ssc:combined-norm}, we prove that it is equivalent 
to the distance to the exact solution measured in terms of the combined (primal-dual) 
norm.
  
An advanced form of the majorant (which is 
sharper than those in Theorems \ref{th:theorem-minimum-of-majorant-I},  
\ref{th:theorem-majorant-for-decomposed-domain-1}, and 
\ref{th:theorem-majorant-for-decomposed-domain-2} but has a more complicated structure)
is derived in Section \ref{sec:majorant-II}.
In Subsection \ref{ssc:equivalence-of-error-and-advanced-majorant}, 
it is shown that the advanced majorant is equivalent to the distance to the exact 
solution measured in terms of the primal energy norm.
A guaranteed and fully computable lower bound of the error is derived in Theorem 
\ref{th:theorem-mininum-of-minorant}. The minorant (\ref{eq:lower-estimate})
also contains only known data and can be computed directly. Finally, we note that 
the practical efficiency of estimates similar to those derived in this paper has been recently 
tested and confirmed in \cite{MatculevichRepin2013}. 


\section{Majorants of the deviation from \texorpdfstring{$\boldsymbol{u}$}.}
\label{sec:majorant-I}

In this section, we deduce the first (and the simplest) form of the functional, which
provides a guaranteed and fully computable upper bound of the deviation (error) 
$e = u - v$ for any function \linebreak 
$v\in \Ho{1}(Q_T)$ and the solution $u$. From 
(\ref{eq:generalized-statement}), it follows that
\begin{multline}
	\Int_{\Omega} \! \! \left( e(x, T) \eta(x, T) - e(x, 0) \eta(x, 0) \right) \! \dx - \!
	\Int_{Q_T} \! \! e \eta_t  \! \dxt \, + 
	\Int_{Q_T} \! \! A \nabla{e} \cdot \nabla{\eta}  \! \dxt \, + 
	\Int_{Q_T} \! \! \lambda e \eta \! \dxt \, + \qquad \\ 
	\Int_{S_R} \sigma e \eta \dst = 
	\Int_{Q_T} \! \left(f - v_t - \lambda v \right)\eta \dxt - 
	\Int_{Q_T} \! A \nabla{v} \cdot \nabla{\eta} \dxt + \Int_{S_R} (g - \sigma v) \eta \dst.
	\nonumber
\end{multline}
Since $e \in \Ho{1}(Q_T)$, we can set $\eta = e$, use the relation
\begin{equation}
	\Int_{\Omega} \left( e^2(x, T) - e^2(x, 0)\right) \dx - \Int_{Q_T} e e_t \dxt = 	
	\frac12 \left(\| e(\cdot, T) \|^2_{\Omega} - \| e(\cdot, 0) \|^2_{\Omega} \right),
	\label{eq:e-et-relation}
\end{equation}
and obtain
\begin{multline}
	\frac12 \| e(\cdot, T) \|^2_{\Omega} + \NormQT{ \nabla{e} } + 
	\Int_{Q_T} \!\! \lambda e^2 \dxt + \Int_{S_R} \! \sigma e^2 \dst = 
	\Int_{Q_T} \!\! \left(f - v_t - \lambda v \right) e \dxt - \\
	\Int_{Q_T} A \nabla{v} \cdot \nabla e \dxt + 
	\Int_{S_R} ( g - \sigma v ) e \dst + \frac12 \| e(\cdot	, 0) \|^2_{\Omega} \,.
	\label{eq:energy-balance-equation}
\end{multline}
This relation is a form of the `energy-balance' identity in terms of deviations. 
It plays an important role in subsequent analysis. Next, we introduce an additional 
variable $y \in Y^*_{\dvrg}(Q_T)$, where
\begin{equation}
	Y^*_{\dvrg}(Q_T) := 
	\bigg \{  \, y \in L_2(\Omega) \,\Big|\, \dvrg y \in L^2(\Omega), \;
	 y \cdot n \in L^2(\Gamma_R)\; \mbox{for a.a.} \; t \in (0, T) \bigg \}.
	\label{eq:y-set-div}
\end{equation} 	
%
%
\begin{theorem}
\label{th:theorem-minimum-of-majorant-I}
(i) For any $v \in \Ho{1}(Q_T)$ and $y \in Y^*_{\dvrg}(Q_T)$ the following inequality 
holds:
\begin{multline}
	(2 - \delta)\NormA{\nabla e} \! + 
	\left( \! 2 - \frac1\gamma \! \right) \left \| \sqrt{\lambda} e  \right \|^2_{Q_T}  + 
	\| e (\cdot, T) \|^2_{\Omega} \! + 2 \left \| \sqrt{\sigma} e  \right \|^2_{S_R} =: \\
	\error_{({\nu},\, {\theta},\, {\zeta}, \,  {\chi})} \leq 
	\maj (v, y; \delta, \gamma, \mu) \! := 
	\| e (\cdot, 0) \|^2_{\Omega}\, +  
	\qquad \qquad \qquad \qquad \\
	        \Int_0^T \! \! \Bigg ( \! \gamma \left\|\frac{\R_{f,\, \mu} (v, y)}{\sqrt{\lambda} \, } \right\|^2_{\Omega} \, + 
	                 \alpha_1(t) \frac{\CF^2}{\nu_1} \| \R_{f, 1 - \mu} (v, y) \|^2_{\Omega} + \\
									 \alpha_2(t) \| \R_d (v, y)\|^2_{A^{-1}} + 
									 \alpha_3(t) \frac{\Ctr^2}{\nu_1} \left\| \R_b (v, y) \right\|^2_{\Gamma_R} \Bigg ) \dt,
  \label{eq:majorant-1}
\end{multline}
where $\delta \in (0, 2]$, $\gamma \geq 1$, $\mu \in [0, 1]$,
\begin{alignat}{2}
\R_f  (v, y) &        := f - v_t - \lambda v + \dvrg y, \\
\R_{f,\,\mu} (v, y) & := \mu \, \R_f, \quad \R_{f, 1 - \mu}  (v, y) := (1 - \mu) \, \R_f, \label{eq:r-mu} \\
\R_d  (v, y) &        := y - A \nabla{v}, \label{eq:r-d} \\ 
\R_b  (v, y) &        := g - \sigma v - y \cdot n, 
\end{alignat}
$\CF$ is the constant in the Friedrichs' inequality
\begin{equation}
	\| \eta \|_{\Omega} \leq \CF \| \nabla \eta \|_{\Omega}, \qquad \forall \eta \in \Ho{1}(\Omega), 
	\label{eq:friedrichs}
\end{equation}
$\Ctr$ is the constant in the trace inequality related to the Robin part of the boundary
\begin{equation}
\| \eta \|_{\Gamma_R} \leq \Ctr \| \nabla \eta \|_{\Omega}, \qquad \forall \eta \in \Ho{1}(\Omega), 
\end{equation} 
${\nu} = 2 - \delta$, 
${\theta} = \sqrt{ \left( 2 - \frac1\gamma \right) \lambda }$, 
${\zeta} = 1$,
${\chi} = 2$
are positive weights, and $\alpha_1(t)$, $\alpha_2(t)$, $\alpha_3(t)$ are positive scalar-valued functions satisfying the relation
\begin{equation}
	\frac{1}{\alpha_1(t)} + \frac{1}{\alpha_2(t)} +
	\frac{1}{\alpha_3(t)} = \delta.
	\label{eq:alpha-relation}
\end{equation}

\noindent
(ii) For any $\delta \in (0, 2]$, $\gamma \geq 1$, and $\mu \in [0, 1]$, the lower bound
of the variation problem generated by the majorant
\begin{equation}
\inf\limits_{
\begin{array}{c}
v \in \Ho{1}(Q_T)\\
y \in Y^*_{\dvrg}(Q_T)
\end{array}
} \maj (v, y; \delta, \gamma, \mu)
\label{eq:inf-maj-I}
\end{equation}
is zero, and it is attained if and only if $v = u$ and $y = A \nabla u$.
\end{theorem}
%
%
\begin{proof}
(i) We  transform the right-hand side of (\ref{eq:energy-balance-equation}) by 
means of the relation 
\begin{equation*}
	\Int_{Q_T} \dvrg y \: \eta \dxt + \Int_{Q_T} y \cdot \nabla{\eta}\dxt =
	\Int_{S_R} y \cdot n \: \dst,
\end{equation*}
which yields
\begin{multline}
	\frac12 \| e(\cdot, T) \|^2_{\Omega} + \NormA{\nabla{e}} + 
	\Int_{S_R} \sigma e^2 \dst + 
        \Int_{Q_T} \lambda e^2 \dxt = \\
	\I_f + \I_d + \I_b + \frac12 \| e(\cdot, 0) \|^2_{\Omega},
	\label{eq:gen-stat-u-v-norm-y}
\end{multline}
where
\begin{equation}
	\I_f = \Int_{Q_T} \R_f \, e \dxt, \quad
	\I_d = \Int_{Q_T} \R_d  \cdot \nabla{e} \dxt, \quad 
	\I_b = \Int_{S_R} \R_b \, e \dst. 
	\label{eq:Ir-Id-Ib-terms}
\end{equation}
%
By means of the H\"older inequality, we find that 
\begin{equation}
	\I_d = \Int_{Q_T} \R_d \cdot \nabla{e} \dxt \leq
	       \Int_0^T   \left  \| \, \R_d \, \right \|_{A^{-1}} \|\nabla{e}\|_A \dt
	\label{eq:holder-estimate-1}			
\end{equation}	
and
\begin{equation}
	\I_b = \Int_{S_R} \R_b \, e \dst \leq 
	      \Int_0^T \left\| \R_b \right\|_{\Gamma_R} \|e\|_{\Gamma_R} \dt \leq 
 	      \Int_0^T \left\| \R_b \right\|_{\Gamma_R} \frac{\Ctr}{\sqrt{\nu_1}} \|\nabla{e}\|_A \dt, 
	\label{eq:holder-estimate-2}
\end{equation}	
where $\nu_1$ is the constant in (\ref{eq:operator-a}). 
Let $\mu(x,t)$ be a real-valued function taking values in $[0, 1]$. Next, we estimate 
the term $\I_f$ as follows:
\begin{equation}
	\I_f \leq \bigintsss\limits_0^T \Bigg (
	         \bigg\| \,  \frac{\R_{f,\, \mu}}{\sqrt{\lambda}} \,  \bigg\|_{\Omega}
	         \big \|  \sqrt{\lambda}e  \,  \big \|_{\Omega} +
	         \frac{\CF}{\sqrt{\nu_1}} \Big \| \,  \R_{f, 1 - \mu} \,  \Big \|_{\Omega}
	         \|\, \nabla e \, \|_A \Bigg )\dt.
  \label{eq:holder-estimate-3}
\end{equation}	
In \cite{RepinSauter2006}, this decomposition was used in order to overcome
difficulties arising in the stationary problem  if $\lambda$ is small (or zero) 
in some parts of the domain and large in another (more detailed study of this 
form of the majorant can be found in 
\cite{NeittaanmakiRepinMaxwell2010} and  \cite{MaliNeittaanmakiRepin2013}).

By combining (\ref{eq:holder-estimate-1})--(\ref{eq:holder-estimate-3}), we obtain
\begin{multline}
	\frac12 \| e(\cdot, T) \|^2_{\Omega} + \NormA{\nabla e} + 
	\Int_{S_R} \sigma e^2 \dst + 
        \Int_{Q_T} \lambda e^2 \dxt \leq 
	\frac12 \| e(\cdot, 0) \|^2_{\Omega} + \\
	\Int_0^T \Bigg (
	\bigg\| \, \frac{\R_{f,\, \mu}}{\sqrt{\lambda}} \, \bigg\|_{\Omega} \big \|\sqrt{\lambda} \, e \big \|_{\Omega} + 
        \frac{\CF}{\sqrt{\nu_1}} \|\R_{f, 1 - \mu} \,\|_{\Omega} \|\nabla e \|_A + \\
	\|\R_d \, \|_{A^{-1}} \|\nabla e\|_A + \big\| \R_b \, \big\|_{\Gamma_R} \frac{\Ctr}{\sqrt{\nu_1}} \|\nabla e\|_A \Bigg ) \dt.
	\label{eq:estimate}
\end{multline} 	      	      
The second term in the right-hand side of (\ref{eq:estimate}) is estimated by the
Young--Fenchel inequality
\begin{alignat}{2}	
	\Int_0^T \bigg\| \frac{\R_{f,\, \mu}}{\sqrt{\lambda}}  \,\bigg\|_{\Omega} 
		 \big \| \sqrt{\lambda} \, e \big \|_{\Omega} \dt & \leq   
	\Int_0^T \Bigg ( \frac\gamma2 \bigg\| \frac{\R_{f,\, \mu}}{\sqrt{\lambda}}  \,\bigg\|^2_{\Omega} +
		 \frac{1}{2\gamma} \big \| \sqrt{\lambda} e \big \|^2_{\Omega} \Bigg ) \dt, 
	\label{eq:young-fenchel-1}
\end{alignat}	
where $\gamma$ is an arbitrary positive constant parameter. Analogously,
\begin{alignat}{2}	
	\Int_0^T \frac{\CF}{\sqrt{\nu_1}} \| \R_{f, 1 - \mu}\|_A \|\nabla e \|_A \dt & \leq 
	\Int_0^T \! \! \Bigg( \! \frac{\alpha_1(t)}{2}  \frac{\CF^2}{\nu_1} \|\R_{f, 1 - \mu} \|^2_{\Omega} \! + 
	               \frac{1}{2\alpha_1(t)} \|\nabla e \|^2_A \! \Bigg) \dt, \\
	\Int_0^T \| \R_d \, \|_{A^{-1}} \|\nabla e\|_A  \dt & \leq 
  \Int_0^T \! \! \Bigg( \! \frac{\alpha_2(t)}{2}  \| \R_d \, \|^2_{A^{-1}} \! + 
							    \frac{1}{2\alpha_2(t)} \|\nabla e\|^2_A \! \Bigg) \dt,
\end{alignat}	
and
\begin{alignat}{2}
	\Int_0^T \big\| \R_b \big\|_{\Gamma_R}
	               \frac{ \Ctr }{\sqrt{\nu_1}} \|\nabla e\|_A & \leq 
	\Int_0^T \Bigg(\frac{\alpha_3(t)}{2} \frac{\Ctr^2}{\nu_1} \| \R_b \|^2_{\Gamma_R} + 
							  \frac{1}{2\alpha_3(t)} \|\nabla e \|^2_A \Bigg) \dt.
	\label{eq:young-fenchel-4}
\end{alignat}
Here, $\alpha_1(t)$, $\alpha_2(t)$, and $\alpha_3(t)$ are functions satisfying 
the relation (\ref{eq:alpha-relation}). Then, the estimate (\ref{eq:majorant-1}) follows 
from (\ref{eq:young-fenchel-1})--(\ref{eq:young-fenchel-4}).

(ii) Existence of the pair $(v, y) \in \Ho{1}(Q_T) \times Y_{\dvrg}^*(Q_T)$ minimizing 
the functional $\maj (v, y; \delta, \gamma, \mu)$ can be proven 
straightforwardly. Indeed, let $v = u$ and
$y = A \nabla u$. Since $\dvrg (A \nabla u) \in \L{2}(Q_T)$, we see that 
$y \in Y_{\dvrg}^*(Q_T)$. In this case, 
according to 
(\ref{eq:parabolic-problem-equation})--(\ref{eq:parabolic-problem-robin-boundary-condition}),
\begin{alignat}{2}
e(x, 0) & = (u - v)(x, 0) = \varphi(x) - v(x, 0) = 0, \nonumber\\
\R_f(u, A \nabla u) & = f - u_t - \lambda u + \dvrg A \nabla u = 0, \nonumber\\
\R_d (u, A \nabla u) & = A \nabla u - A \nabla{u} = 0, \nonumber \\
\R_b (v, y) & =  g - \sigma v - A \nabla u \cdot n = 0, \nonumber
\end{alignat}
Thus, we see that $\maj ( u, A \nabla u ; \delta, \gamma, \mu) = 0$ and, therefore, 
the exact lower bound of 
$\maj ( v, y ; \delta, \gamma, \mu)$ is attained on the pair presenting the exact solution of 
(\ref{eq:parabolic-problem-equation})--(\ref{eq:parabolic-problem-robin-boundary-condition}).

Assume that $\maj (v, y; \delta, \gamma, \mu) = 0$, which means that for a.a. 
$(x, t) \in Q_T$ the following relations hold:
\begin{alignat}{4}
	y = A \nabla v                    \; & \quad \mbox{a.a.} & \quad (t, x) \in Q_T, 
	\label{eq:requirements-set-1}\\
	f - v_t - \lambda v + \dvrg y = 0 \; & \quad \mbox{a.a.} & \quad (t, x) \in Q_T, 
	\label{eq:requirements-set-2}\\
  v(x, 0) = \varphi(x)                     \; & \quad \mbox{a.a.} & \quad  x \in \Omega, 
	\label{eq:requirements-set-3}\\
  v = 0                       \; & \quad \mbox{a.a.} & \quad (t, x) \in S_D, 
	\label{eq:requirements-set-4}\\
	y \cdot n + \sigma v = g          \; & \quad \mbox{a.a.} & \quad (t, x) \in S_R.
  \label{eq:requirements-set-5}
\end{alignat} 
From (\ref{eq:requirements-set-2})--(\ref{eq:requirements-set-5}), it follows that
\begin{equation}
	\Int_{Q_T} ( f - v_t - \lambda v ) \eta \dxt - \Int_{Q_T} y \cdot \nabla{\eta} + 
	\Int_{S_N} g \eta \dst = 0,
	\quad \forall \eta \in \Ho{1}(Q_T).
	\label{eq:generalized-statement-for-v}
\end{equation}
%
%
In view of (\ref{eq:requirements-set-1}), this relation is equivalent to 
(\ref{eq:generalized-statement}), whence it follows that $v = u$ and 
$y = A \nabla u$.
\end{proof}
%
%
\begin{remark}

We see that $\maj (v, y; \delta, \gamma, \mu)$ depends on a collection of 
parameters, which can be selected within certain admissible sets. Varying $\delta$ and 
$\gamma$ allows us to obtain estimates for different error measures. By selecting the 
functions $\alpha_i$ and $\mu$, we find the best possible value of the majorant. This fact 
is beneficial for practical applications because we can select values of the parameters 
in an optimal way for a concrete problem. In particular, $\mu$ can be set to $0$ and $1$. For
these two cases, we use the abridged notation $\majmuzero$ and $\majmuone$:
\begin{multline*}
	\majmuzero \! := 
	\| e (\cdot, 0) \|^2_{\Omega} + 
	                 \Int_0^T \! \! \Bigg ( \! \alpha_1(t) \frac{\CF^2}{\nu_1} \| \R_f \|^2_{\Omega} + 
									 \alpha_2(t) \| \R_d\|^2_{A ^{-1}} + 
									 \alpha_3(t) \frac{\Ctr^2}{\nu_1} \left\| \R_b \right\|^2_{\Gamma_R} \!\! \Bigg)\dt
\end{multline*}
and
\begin{multline*}
	\majmuone  \! := 
	\| e (\cdot, 0) \|^2_{\Omega} + \!
	        \Int_0^T \! \! \Bigg ( \! \gamma \left\|\frac{\R_f }{\sqrt{\lambda}} \, \right\|^2_{\Omega} \! + 
	                 \alpha_2(t) \| \R_d \|^2_{A^{-1}} + 
									 \alpha_3(t) \frac{\Ctr^2}{\nu_1} \left\| \R_b \right\|^2_{\Gamma_R} \! \Bigg ) \dt.
	\label{eq:majorant-mu-1}
\end{multline*}
The majorant $\majmuzero$ is well adapted to 
problems, in which $\lambda$ is small or zero (so that the impact of the reaction term 
is insignificant). In such type problems, we should avoid the term 
$\left\|\dfrac{\R_{f,\, \mu} (v, y) }{\sqrt{\lambda}} \, \right\|^2_{\Omega}$, which makes the 
whole estimate sensitive to the residual $\R_f(v, y)$ and may lead to a considerable 
overestimation of the error. The estimate $\majmuone$ 
is useful if $\lambda$ is not small and may attain large 
values in some parts of $\Omega$. If $\lambda$ reaches both small (or zero) and large 
values, then the combined estimate (\ref{eq:majorant-1}) is preferable. 
\end{remark}


\subsection{Estimates based upon domain decomposition}

The majorant defined by (\ref{eq:majorant-1}) contains the Friedrichs constant $\CF$ and
the trace constant $\Ctr$. If $\Omega$ has a complicated geometry, then finding these 
constants (or 
guaranteed bounds of them) may not be an easy task. Below we suggest the method, 
which allows to overcome this difficulty. It is based on domain decomposition and leads 
to the estimates with a different 
set of constants (a consequent discussion of this method for elliptic problems can be 
found in \cite{RepinDeGruyter2008}).  

Assume that $\Omega$ is decomposed into a set of sub-domains
\begin{equation}
\overline{\Omega}  = \bigcup\limits_{\,i=1, ... ,N} \overline{\Omega}_i, 
\quad \Omega_i \cap \Omega_j = \emptyset, 
\quad i \neq j. 
\label{eq:omega-representation}
\end{equation}  
We use the Poincare inequalities
\begin{equation}
\left \| \, \widetilde{w} \,\right\|_{\Omega_i} \leq 
\CPi \big \| \nabla w \big \|_{\Omega_i}, \quad i = 1, ... , N, \quad \forall w \in \H{1}(\Omega),
\label{eq:poincare-inequality}
\end{equation}
where $\widetilde{w} = w - {{ \big \{ | w} | \big\} }_{\Omega_i}	$, and 
${{ \big \{ | w} | \big\} }_{\Omega_i}$ denotes the 
mean value of $w$ on $\Omega_i$. If all $\Omega_i$ are convex, then $\CPi$ can be estimated from the 
above by the quantity ${\rm diam} \, \Omega_i/\pi$ (see \cite{PayneWeinberger1960}). We 
use this fact in order to represent the majorant in a somewhat different form. 
In further analysis, we assume (for the sake of simplicity only) that $S_T = S_D$ and 
$\varphi(x) = v(x, 0)$.

\begin{theorem}
\label{th:theorem-majorant-for-decomposed-domain-1}
\vskip 5pt
For any $v \in \Ho{1}(Q_T)$ and $y \in Y^*_{\dvrg}(Q_T)$ the following inequality holds:
\begin{multline}
	(2 - \delta)\NormA{\nabla e} + 
	\Bigg( 2 - \frac{1}{\rho_1} - \frac{1}{\rho_2} \Bigg) \Big \| \sqrt{\lambda} e  \Big \|^2_{Q_T}  + 
	\| e (\cdot, T) \|^2_{\Omega} =: \error_{({\nu},\, {\theta},\, {\zeta})} \leq \majd \! := \\
	        \Int_0^T \! \! \Bigg ( \! \rho_1 \! \left \| \frac{\R_{f,\, \mu} (v, y)}{\sqrt{\lambda}}  \, \right \|^2_{\Omega} \!\! + 
					\rho_2 R^2_{{\mathrm I}, 1}(t) \! + 
					\alpha_1(t) R^2_{{\mathrm I}, 2}(t) \! + 
					\alpha_2(t)  \| \R_d (v, y) \, \|^2_{A^{-1}} \! \Bigg ) \! \dt, 
\label{eq:majorant-decomposed-ii}
\end{multline}
where $\delta \in (0, 2]$, 
$\rho_1 \geq \dfrac{1}{2 - \frac{1}{\rho_2}} $, 
$\mu \in [0, 1]$, $\R_{f,\, \mu}(v, y)$ and $\R_d(v, y)$ are defined in
(\ref{eq:r-mu}) and (\ref{eq:r-d}), respectively, and  
\begin{equation*}
R_{{\mathrm I}, 1}(t) := \sqrt{\Sum_{i = 1}^{N} \frac{|\Omega_i|}{\lambda_i} \left( \! \Mean{ \R_{f, 1 - \mu} }{\Omega_i} \! \right)^2 }, \;
R_{{\mathrm I}, 2}(t) :=  \sqrt{\! \Sum_{i = 1}^{N} \frac{\CPi^2}{\nu_1} \left\| \R_{f, 1 - \mu} \right\|^2_{\Omega_i}  }.
\end{equation*}
Here, $\lambda_i = \min\limits_{x \in \Omega_i} \lambda(x, t)$ for a.a. $t \in [0, T]$, 
${\nu} = 2 - \delta$, 
${\theta} =\sqrt{ \left(2 - \frac{1}{\rho_1} - \frac{1}{\rho_2} \right) \lambda}$, and
${\zeta} = 1$, and $\alpha_1(t)$, $\alpha_2(t)$ are 
positive scalar-valued functions satisfying the relation
$\frac{1}{\alpha_1(t)} + \frac{1}{\alpha_2(t)} = \delta$.
\end{theorem}

\begin{proof}
Consider the integral identity (\ref{eq:gen-stat-u-v-norm-y}). 
The term $\I_f$ can be represented as 
\begin{equation}
\I_f = \Int_{Q_T}  \R_{f,\, \mu} e \dxt +  \Int_{Q_T} \R_{f, 1 - \mu} e \dxt 
     = \I^{{\mu}}_f + \I^{1 - {\mu}}_f.
\label{eq:representation-of-if-with-mu}
\end{equation}
$\I^{{\mu}}_f$ is estimated as 
\begin{equation}
\I^{{\mu}}_f \leq  \Int_0^T\left \| \frac{\R_{f,\, \mu}}{\sqrt{\lambda}} \right \|_{\Omega}  
          \left \| \sqrt{\lambda} e \right \|_{\Omega} \dt.
\label{eq:estimate-of-if-with-mu}
\end{equation}
By means of the H\"older inequality, for $\I^{1 - {\mu}}_f$ we have
\begin{multline}
 \I^{1 - {\mu}}_f = \Int_0^T \Bigg( \Sum_{i = 1}^{N} \Int_{\Omega_i} \widetilde{\R}_{f, 1 - \mu} \, e \dx + 
\Sum_{i = 1}^{N} \Mean{ \R_{f, 1 - \mu} }{\Omega_i} \Int_{\Omega_i} e \dx \Bigg) \dt \leq \\ 
\Int_0^T  \Sum_{i = 1}^{N} \Int_{\Omega_i}  \widetilde{\R}_{f, 1 - \mu} \, e \dx \dt  + 
\Int_0^T  \Sum_{i = 1}^{N} \frac{\sqrt{|\Omega_i|}}{\sqrt{\lambda_i}} \Mean{  \R_{f, 1 - \mu} }{\Omega_i}  \big \|\sqrt{\lambda} \, e \big \|_{\Omega_i}  \dt. 
\label{eq:majorant-on-decomposed-omega}
\end{multline} 
where $\lambda_i = \min\limits_{x \in \Omega_i} \lambda(x, t)$ for a.a. $t \in [0, T]$.
Each of the terms on the right-hand side of (\ref{eq:majorant-on-decomposed-omega}) can 
be estimated as follows:
\begin{alignat}{2}
\Int_0^T & \Sum_{i = 1}^{N} \; \Int_{\Omega_i} \widetilde{\R}_{f, 1 - \mu} e \dx \dt \leq 
\Int_0^T {R}_{{\mathrm I}, 2} \; \| \nabla e \|_{A} \dt, \\ 
\Int_0^T & \Sum_{i = 1}^{N}  \frac{\sqrt{|\Omega_i|}}{\sqrt{\lambda_i}}\Mean{ \R_{f, 1 - \mu} }{\Omega_i}  \big \|\sqrt{\lambda} \, e \big \|_{\Omega_i} \dt \leq \Int_0^T {R}_{{\mathrm I}, 1}  \big \|\sqrt{\lambda} \, e \big \|_{\Omega} \dt.
\end{alignat}
At last, using the Young--Fenchel inequality, we obtain the following estimates
\begin{alignat}{2}
\Int_0^T \left \| \frac{\R_{f,\, \mu} }{\sqrt{\lambda}} \right \|_{\Omega} \left \| \sqrt{\lambda} e \right \|_{\Omega} \dt & \leq 
\Int_0^T \Bigg(
	\frac{{\rho}_1}{2} \left \| \frac{\R_{f,\, \mu}}{\sqrt{\lambda}}  \right \|^2_{\Omega} + 
	\frac{1}{2{\rho}_1} \left \| \sqrt{\lambda} e \right \|^2_{\Omega} 
\Bigg) \dt, \label{eq:young-fenchel-for-decomposed-domain-2-ii-1}
\end{alignat}
\begin{alignat}{2}
\Int_0^T {R}_{{\mathrm I}, 1} \|\sqrt{\lambda} \, e \big \|_{\Omega} \dt & \leq 
  \Int_0^T \Bigg( \frac{{\rho}_2}{2} {R}^2_{{\mathrm I}, 1} + 
	\frac{1}{2 {\rho}_2} \big \|\sqrt{\lambda} \, e \big \|^2_{\Omega} 
  \Bigg) \dt,\label{eq:young-fenchel-for-decomposed-domain-2-ii-2} \\
\Int_0^T {R}_{{\mathrm I}, 2} \; \| \nabla e \|_{A} \dt & \leq \Int_0^T \Bigg(
	\frac{{\alpha}_1(t)}{2} {R}^2_{{\mathrm I}, 2} + \frac{1}{2 {\alpha}_1(t)} \| \nabla e \|^2_{A},
\Bigg) \dt, \label{eq:young-fenchel-for-decomposed-domain-2-ii-3}
\end{alignat}
and, analogously, 
\begin{multline}
	\Int_0^T \| \R_d (v, y) \|_{A^{-1}} \|\nabla e\|_A  \dt \leq 
  \Int_0^T \Bigg( \frac{{\alpha}_2(t)}{2}  \| \R_d \|^2_{A^{-1}} + 
							    \frac{1}{2{\alpha}_2(t)} \|\nabla e\|^2_A \Bigg) \dt.
\label{eq:young-fenchel-for-decomposed-domain-2-ii}
\end{multline}
%
By combining 
(\ref{eq:young-fenchel-for-decomposed-domain-2-ii-1})--(\ref{eq:young-fenchel-for-decomposed-domain-2-ii}),
we obtain (\ref{eq:majorant-decomposed-ii}).
\end{proof}

\vskip 10pt

Consider a special case, which arises if we impose additional conditions, namely,
\begin{equation}
\Mean{ \R_{f, 1 - \mu} (v, y) }{\Omega_i} = 0, \quad i = 1, ..., N, \quad
\mbox{for a.a.}\; t \in [0, T], 
\label{eq:mean-condition}
\end{equation}
where $\mu$ is inherited from (\ref{eq:majorant-1}). Since the functions $y$ and $\mu$ 
are in our disposal, these integral type conditions do not lead to essential technical
difficulties provided that $N$ is not too large. Now, (\ref{eq:majorant-decomposed-ii}) 
can be represented in a simpler form.
%
\begin{theorem}
\label{th:theorem-majorant-for-decomposed-domain-2}
\vskip 5pt
If (\ref{eq:mean-condition}) is satisfied, then for any $v \in \Ho{1}(Q_T)$ and 
$y \in Y^*_{\dvrg}(Q_T)$
\begin{multline}
	(2 - \delta)\NormA{\nabla e} \! + 
	\Bigg( \!  2 - \frac{1}{\gamma} \! \Bigg) \Big \| \sqrt{\lambda} e  \Big \|^2_{Q_T} \! + 
	\| e (\cdot, T) \|^2_{\Omega} =: \error_{({\nu},\, {\theta},\, {\zeta})} \leq \\ 
	\majd \, :=
	        \Int_0^T \Bigg ( \gamma \bigg\|\frac{\R_{f,\, \mu} (v, y)}{\sqrt{\lambda}} \, \bigg\|^2_{\Omega} + 
					{\alpha_1(t)} R^2_{\mathrm I}(t) + 
					{\alpha_2(t)}  \| \R_d(v, y) \, \|^2_{A^{-1}} \Bigg) \dt,
\label{eq:majorant-decomposed-i}
\end{multline}
where $\delta \in (0, 2]$, $\gamma \geq \frac12$, $\mu \in [0, 1]$,
$\R_{f,\, \mu}(v, y)$ and $\R_d(v, y)$ are defined in (\ref{eq:r-mu}) and (\ref{eq:r-d}), 
respectively, and 
\begin{equation*}
R_{\mathrm I}(t) :=  \sqrt{ \Sum_{i = 1}^{N} \frac{\CPi^2}{\nu_1}  \left\| \R_{f, 1 - \mu} \, \right\|^2_{\Omega_i}  },
\end{equation*}
${\nu} = 2 - \delta$, 
${\theta} = \sqrt{ \left( 2 - \frac{1}{\gamma} \right) \lambda }$,  and ${\zeta} = 1$ 
are positive weights, and $\alpha_1(t)$, $\alpha_2(t)$ are positive scalar-valued 
functions satisfying the relation
$\frac{1}{\alpha_1(t)} + \frac{1}{\alpha_2(t)} = \delta$.
\end{theorem}

\begin{proof}
If (\ref{eq:mean-condition}) holds, then,
\begin{equation}
 \I^{1 - {\mu}}_f \!=\! \Int_0^T \Sum_{i = 1}^{N} \Int_{\Omega_i} \R_{f, 1 - \mu} \, e \dx \dt \!=\!
                  \Int_0^T \Sum_{i = 1}^{N} \Int_{\Omega_i} \R_{f, 1 - \mu} \, \widetilde{e} \dx \dt.
\end{equation} 
Therefore, using (\ref{eq:poincare-inequality}), we obtain
\begin{equation}
\I^{1-{\mu}}_f \leq \Int_0^T {R}_{\mathrm I} \; \| \nabla e \|_{A} \dt.
\end{equation}
By means of the Young--Fenchel inequality, we deduce
\begin{alignat}{2}
\Int_0^T \left \| \frac{\R_{f,\, \mu}}{\sqrt{\lambda}}  \, \right \|_{\Omega} \left \| \sqrt{\lambda} e \right \|_{\Omega} \dt & \leq
\Int_0^T \left( \frac{{\gamma}}{2} \left \| \frac{\R_{f,\, \mu} }{\sqrt{\lambda}} \, \right \|^2_{\Omega} + 
                \frac{1}{2 {\gamma}} \left \| \sqrt{\lambda} e \right \|^2_{\Omega} \right) \dt 
\label{eq:young-fenchel-for-decomposed-domain-i-1}
\end{alignat}
and
\begin{alignat}{2}
\Int_0^T {R}_{\mathrm I} \; \| \nabla e \|^2_{A} \dt & \leq \Int_0^T \left( \frac{{\alpha}_1(t)}{2} R^2_{\mathrm I} + \frac{1}{2 {\alpha}_1(t)} \| \nabla e \|^2_{A} \right) \dt.
\label{eq:young-fenchel-for-decomposed-domain-i-2}
\end{alignat}
%
The term $\I_d$ is estimated analogously to the method used in proof of Theorem 
\ref{th:theorem-minimum-of-majorant-I}:
\begin{multline}
	\I_d \leq \Int_0^T \| \R_d \,\|_{A^{-1}} \|\nabla e\|_A  \dt \leq 
  \Int_0^T \Bigg( \frac{{\alpha}_2(t)}{2}  \| \R_d \,\|^2_{A^{-1}} + 
							    \frac{1}{2{\alpha}_2(t)} \|\nabla e\|^2_A \Bigg) \dt.
\label{eq:young-fenchel-for-decomposed-domain-i-3}
\end{multline}
Therefore, 
(\ref{eq:young-fenchel-for-decomposed-domain-i-1})--(\ref{eq:young-fenchel-for-decomposed-domain-i-3}) 
yield the estimate (\ref{eq:majorant-decomposed-i}).
\end{proof}

\subsection{Two sided estimates for combined norms}
\label{ssc:combined-norm}

In modern numerical methods (e.g., in various mixed finite element schemes) the 
approximations are generated for both primal and dual components of the solution. We 
note that this concept is perfectly motivated by physical arguments because primal 
and dual components often reflect physically meaningful parts of the solution. By 
following this idea, we now consider the solution of 
(\ref{eq:parabolic-problem-equation})--(\ref{eq:parabolic-problem-robin-boundary-condition})
as a pair $(u, p) \in V^{1, 0}_2 (Q_T) \times Y^*_{\dvrg}(Q_T)$. In order to 
measure the deviation of the approximation 
$(v, y) \in \Ho{1} (Q_T) \times Y^*_{\dvrg}(Q_T)$ from $(u, p)$, we use the combined 
primal-dual norm
\begin{multline}
\left \| [(u, p) - (v, y)] \right \|^2_{(\check{\nu}, \check{\theta}, \check{\zeta}, \check{\chi})} := 
\check{\nu}    \NormA{ \, \nabla{e}} + 
\check{\theta} \NormAinverse{ \, (y - p) } +  \\
\check{\zeta}  \left \| \, \dvrg (p - y) - (u - v)_t \right\|^2_{Q_T} + 
\check{\chi}   \| \, e (\cdot, T) \|^2_{\Omega}.
\label{eq:combined-norm}
\end{multline}
Let $\lambda = 0$, $S_N = S_D$, and $\varphi(x) = v(x, 0)$. Then, from Theorem 
\ref{th:theorem-minimum-of-majorant-I} (with $\beta = {\rm const}$, $\delta = 1$, and 
$\mu = 0$) the estimate can be written in the form

\begin{multline}
	\NormA{\nabla e} + \| e (\cdot, T) \|^2_{\Omega} \leq 
	\maj \, := \\ 
	(1 + \beta) \NormAinverse{  y - A \nabla{v} }	 +
	                            \Bigg(1 + \frac{1}{\beta} \Bigg) \frac{\CF^2}{\nu_1} \| f - v_t + \dvrg y\|^2_{Q_T}.
	\label{eq:majorant-simplified-combined-norm}
\end{multline}
%
Since $p = A \nabla u$, we reform the right-hand side of 
(\ref{eq:majorant-simplified-combined-norm}) as follows:
\begin{multline}
\maj \, \leq 
(1 + \beta) \left( \NormA{\nabla{(u - v)}} + \NormAinverse{  y - p } \right) +  \\
	                            \hfill \Bigg(1 + \frac{1}{\beta} \Bigg) \frac{\CF^2}{\nu_1} \| f - v_t + \dvrg y\|^2_{Q_T}.
\label{eq:majorant-estimated-by-dual-norm}
\end{multline}
By using (\ref{eq:parabolic-problem-equation}), we find that 
\begin{multline}
(1 + \beta) \! \left( \! \NormA{\nabla{(u - v)}} \! + \! \NormAinverse{  y - p } \right) + 
\left( \! \! 1 \! + \! \frac{1}{\beta} \! \right) \! \frac{\CF^2}{\nu_1} \| f - v_t + \dvrg y\|^2_{Q_T} \leq \quad\\
(1 + \beta) \left( \NormA{\nabla{(u - v)}} + \NormAinverse{  y - p } \right) +  
\left( \! \! 1 \! + \! \frac{1}{\beta} \! \right) 
\frac{\CF^2}{\nu_1} \| \dvrg (p - y) - (u - v)_t\|^2_{Q_T} + \\
\| e (\cdot, T) \|^2_{\Omega}  = 
\big \| [(u, p) - (v, y)] \big \|^2_{(\check{\nu}, \check{\theta}, \check{\zeta}, \check{\chi})},
\end{multline}
where 
$\check{\nu} = \check{\theta} = (1 + \beta)$, 
$\check{\zeta} = \left(1 + \frac{1}{\beta}\right) \frac{\CF^2}{\nu_1}$, and 
$\check{\chi} = 1$.
Next, by combining the first two terms, applying 
(\ref{eq:majorant-simplified-combined-norm}), and, finally, adding and subtracting 
$A \nabla v$ in the third term, we obtain 
\begin{alignat}{2}
& \big \|  [(u, p) - (v, y)] \big \|^2_{(\check{\nu}, \check{\theta}, \check{\zeta}, \check{\chi})} \leq 
\max{ \Big\{ 1, (1 + \beta) \Big \} } \Big( \| e (\cdot, T) \|^2_{\Omega} + \NormA{\nabla{(u - v)}} \Big) + \nonumber\\
& \qquad \quad \,
(1 + \beta) \NormAinverse{  y - A \nabla v + A \nabla v - p } + \Bigg(1 + \frac{1}{\beta} \Bigg) \frac{\CF^2}{\nu_1} \| f - v_t + \dvrg y \|^2_{Q_T} \leq \nonumber\\
& \max{  \Big\{ \! 1,  (1 + \beta) \! \Big \} } \left( 
(1  + \beta) \! \NormAinverse{  y - A \nabla{v} } + 
\left( 1 + \frac{1}{\beta} \right) \frac{\CF^2}{\nu_1} \| f - v_t +  \dvrg y\|^2_{Q_T} \right) \! +  \nonumber\\	
& ( 1 \! + \beta) \!\left( \! \NormAinverse{  y - A \nabla v}\! \! +  \NormAinverse{ A \nabla v - p }\! \! \right) \! + \! 
  \left( \! 1 + \frac{1}{\beta} \!\right) \! \frac{\CF^2}{\nu_1} \| f - v_t + \dvrg y  \|^2_{Q_T}.
\end{alignat}
Hence, we obtain the double inequality
\begin{equation}
\maj \!\leq \!
\big \| [(u, p) - (v, y)] \big \|^2 \! \leq \!
\bigg( \!\! \max{ \Big\{ \! 1, (1 + \beta) \! \Big \} } \! + \! \beta \! + \! 2 \! \bigg) \maj,
\label{eq:double-inequality}
\end{equation}
which shows that the majorant is equivalent to the combined primal-dual error norm. 
In other words, $\maj$ (which contains only known functions and parameters) adequately 
reflects the distance from $(v, y) \in \Ho{1}(Q_T) \times Y^*_{\dvrg}(Q_T)$ to the exact 
solution $(u, p)$. 
In particular, this means that if $(u_h, p_h)$ is the sequence of approximations
computed on a certain set of meshes $\mathcal{F}_h$, which converges to $(u, p)$
with the rate $h^\alpha$, then the values of the majorant
tend to zero with the same rate.


\section{An advanced form of the majorant}
\label{sec:majorant-II}

\begin{theorem}
\label{th:theorem-minimum-of-majorant-II}
(i) For any $v , w \in \Ho{1}(Q_T)$ and $y \in Y_{\dvrg}^*(Q_T)$ the 
following estimate holds:
\begin{multline}
	(2 - \delta)\NormA{\nabla e} + 
	\Bigg( 2 - \frac1\gamma \Bigg) \: \Big \| \sqrt{\lambda} e \Big \|^2_{Q_T} +  
	\Bigg( 1 - \frac1\epsilon \Bigg) \: \| e(\cdot, T) \|^2_{\Omega} + 
	2 \Big \| \sqrt{\sigma}\, e \Big \|^2_{S_R} =: \\
	[ e ]^2_{({\nu},\, {\theta},\, {\zeta}, \,  {\chi})} \leq 
	\majtwo (v, y, w; \delta, \epsilon, \gamma, \mu) := 
	\epsilon \| w(\cdot, T)\|^2_{\Omega} + 
	2L(v, w) + l(w, v)\qquad \\
	\Int\limits_0^T \Bigg (
	\gamma \bigg\|\frac{\R_{1,\, \mu} (v, y, w)}{\sqrt{\lambda}} 
	 \bigg\|^2_{\Omega} + 
	\alpha_1(t) \frac{\CF^2}{\nu_1} \|\, \R_{1, 1 - \mu} (v, y, w)\|^2_{\Omega} + \\
	\qquad \qquad \qquad 
	\alpha_2(t) \| \R_2 (v, y, w) \|^2_{A^{-1}} + 
	\alpha_3(t) \frac{\Ctr^2}{\nu_1} \big\| \R_3 (v, y, w) \big\|^2_{\Gamma_R} 
	\Bigg ) \dt ,	  
\label{eq:majorant-2}
\end{multline}
where $\delta \in (0, 2]$, $\gamma\geq \frac12$, $\epsilon \geq 1$, and 
$\mu \in [0, 1]$,
\begin{equation}
L(v, w) = \Int_{Q_T} \Big( v_t \,w + A \nabla v \cdot \nabla w + \lambda v \,w  - f w \Big) \dxt -
	        \Int_{S_R} (g - \sigma v) \,w \dst, 
\label{eq:l-function}	    
\end{equation}
\begin{equation}
l(v, w) = \Int_\Omega |v(x, 0) - \varphi(x)|^2 - 2 w(x, 0) \big( \varphi(x) - v(0, x)\big) \dx, 
\label{eq:l-small-function}	    
\end{equation}
and
\begin{alignat}{2}
\R_1 (v, y, w) & := f - {(v + w)}_t - \lambda (v - w) + \dvrg y, \label{eq:r-1} \\
\R_{1,\, \mu} (v, y, w) & := \mu \,\R_1 (v, y, w), \quad \R_{1, 1 - \mu} (v, y, w) := (1 - \mu) \, \R_1 (v, y, w), \label{eq:r-1-mu} \\
\R_2 (v, y, w) & := y - A \nabla{(v - w)}, \label{eq:r-2} \\ 
\R_3 (v, y, w) & := g - \sigma (v - w) - y \cdot n,
\end{alignat}
%
${\nu} = 2 - \delta$,
${\theta} = \sqrt{\left(2 - \frac1\gamma\right) \lambda }$, 
${\zeta} = 1 - \frac1\epsilon$,
${\chi} = 2$ are positive weights, and 
$\alpha_1(t)$, $\alpha_2(t)$, and $\alpha_3(t)$ are positive 
function satisfying (\ref{eq:alpha-relation}).

\noindent(ii) For any $\delta \in (0, 2]$, $\gamma\geq \frac12$, $\epsilon \geq 1$, 
and $\mu \in [0, 1]$ the lower bound of the variation problem 
\begin{equation}
\inf\limits_{
\begin{array}{c}
v, w \in \Ho{1}(Q_T)\\
y \in Y^*_{\dvrg}(Q_T)
\end{array}
}
\majtwo (v, y, w; \delta, \epsilon, \gamma, \mu)
\label{eq:inf-maj-II}
\end{equation}
is zero, and it is attained if and only if $v = u$, $y = A \nabla u$, and $w = 0$.
\end{theorem} 
%
\noindent
{\bf Proof:} 
(i) We rewrite the right-hand side of (\ref{eq:energy-balance-equation}) by inserting 
functions $w \in \Ho{1}(Q_T)$ and $y \in Y_{\dvrg}^*(Q_T)$, which implies the following 
relation
\begin{multline}
	\frac12 \| e(\cdot, T) \|^2_{\Omega} + \NormA{\nabla{e}} + 
	\Int_{S_R} \sigma e^2 \dst + \Int_{Q_T} \lambda e^2 \dxt = \\
	\I_1 + \I_2 + \I_3 + \Int_{S_R} ( g - \sigma v - y \cdot n) e \dst + 
	\frac12 \| e(\cdot, 0) \|^2_{\Omega} \,,
	\label{eq:decomposition-for-majorant-II}
\end{multline}
where 
\begin{equation}
	\I_1 = \! \! \Int_{Q_T} \! \! \R_1 \,e \dxt, \;
	\I_2 = \! \! \Int_{Q_T} \! \! \R_2 \cdot \nabla e \dxt,\; 
	\I_3 = \! \! \Int_{Q_T} \! \! \big( (w_t \! - \! \lambda)e \! - 
	                                    \!A \nabla{w} \cdot \nabla{e} \big) \! \dxt. 
\end{equation}
%
%
The term $\I_3$ can be rewritten as
\begin{equation}
\I_3 = L(v, w) + \Int_{\Omega}\Big( e(x, T)w(x, T) - e(x, 0)w(x, 0) \Big)\dx + 
\Int_{S_R} \sigma w \, e \dst.
\label{eq:i-3-term-rewritten}	    
\end{equation}
By combining (\ref{eq:decomposition-for-majorant-II}) and (\ref{eq:i-3-term-rewritten}), 
we obtain
\begin{multline}
	\frac12 \| e(\cdot, T) \|^2_{\Omega} + \NormA{\nabla{e}} + 
	\Int_{S_R} \sigma e^2 \dst + \Int_{Q_T} \lambda e^2 \dxt = 
	\I_1 + \I_2 + L(v, w) + \\
	\Int_{S_R} \R_3 \, e \dst + 
	\Int_{\Omega} e(x, T)w(x, T) \dx + \Int_{\Omega} \Bigg(\frac12 e^2(x, 0) - e(x, 0)w(x, 0) \Bigg)\dx\,,
	\label{eq:decomposition-for-majorant-II-2}
\end{multline}
Using the same technique as in Section \ref{sec:majorant-I}, the right-hand side of 
(\ref{eq:decomposition-for-majorant-II-2}) can be estimated the following way:
\begin{alignat}{2}
	\Int_\Omega e(x, T) \,w(x, T)  \dxt & \leq  
	\frac{1}{2\epsilon} \| e(\cdot, T)\|^2_{\Omega} + 
	\frac{\epsilon}{2} \| w(\cdot, T)\|^2_{\Omega}, \label{eq:ineq-young-fenchel-1}\\
	\Int\limits_0^T \bigg \| \frac{\R_{1,\, \mu}}{\sqrt{\lambda}} \, \bigg\|_{\Omega}  
		                    \big  \|\sqrt{\lambda} \, e \big \|_{\Omega} \dt & \leq  
	\Int_0^T \Bigg ( \frac\gamma2 \Bigg\| \frac{\R_{1,\, \mu}}{\sqrt{\lambda}}  \bigg\|^2_{\Omega} + 
					\frac{1}{2\gamma} \big \|\sqrt{\lambda} e \big \|^2_{\Omega} \Bigg ) \dt, \\
  \! \Int_0^T \frac{\CF}{\sqrt{\nu_1}} \bigg\| \R_{1, 1 - \mu} \bigg\|_{\Omega} \|\nabla e\|_A \!\dt & \leq 
	\! \Int_0^T \!\! \Bigg( \!\! \frac{\alpha_1(t)}{2} \frac{\CF^2}{\nu_1} \bigg\|\R_{1, 1 - \mu} \bigg\|^2_{ \Omega} \! + 
								\frac{1}{2\alpha_1(t)}\|\nabla e\|^2_A \!\! \Bigg) \!\dt, \\
  \Int_0^T \! \| \R_2 \|_{A^{-1}} \|\nabla e\|_A  \dt & \leq \Int_0^T \Bigg(\frac{\alpha_2(t)}{2} \| \R_2 \|^2_{A^{-1}} + 
							                \frac{1}{2\alpha_2(t)} \|\nabla{ e }\|^2_A \Bigg) \dt, \\
  \Int_0^T \big\| \R_3 \big\|_{\Gamma_R} \Ctr \|\nabla e\|_A  & \leq  
	\Int\limits_0^T \Bigg(\frac{\alpha_3(t)}{2}  \frac{\Ctr^2}{\nu_1} 
	   \| \R_3 \|^2_{\Gamma_R} + \frac{1}{2\alpha_3(t)} \|\nabla e\|^2_A \Bigg) \dt,
\label{eq:ineq-young-fenchel-2}
\end{alignat}
%
where $\gamma \geq 1$, $\epsilon \geq 1$, and $\alpha_1(t)$, 
$\alpha_2(t)$, and $\alpha_3(t)$ are functions satisfying 
(\ref{eq:alpha-relation}). Thus, by combination of
(\ref{eq:ineq-young-fenchel-1})--(\ref{eq:ineq-young-fenchel-2}),
we obtain the required estimate (\ref{eq:majorant-2}).

\vskip 3pt

\noindent (ii) This item is proven by the same arguments as in Theorem 
\ref{th:theorem-minimum-of-majorant-I}.
\hfill $\square$
\vskip5pt

\subsection{An advanced majorant based upon domain decomposition}

Now, we deduce an advanced versions of the estimates (\ref{eq:majorant-decomposed-ii}) 
and (\ref{eq:majorant-decomposed-i}). Let (\ref{eq:omega-representation}) hold. 
First, we consider the case where $\lambda$ is not small (or zero). Assume 
(for the sake of simplicity only) that $S_T = S_D$. Then, we have the following result. 

\begin{theorem}
\label{th:theorem-decomposed-domain-majorant-II-1}
For any $v , w \in \Ho{1}(Q_T)$ and $y \in Y_{\dvrg}^*(Q_T)$ we obtain the estimate 
\begin{multline*}
	(2 - {\delta})\NormA{\nabla e} \! + 
	\Bigg( \! 2 - \frac{1}{{\rho}_1} - \frac{1}{{\rho}_2} \! \Bigg) \Big \| \sqrt{\lambda} e  \Big \|^2_{Q_T} \! + 
		\Bigg( \! 1 - \frac{1}{{\epsilon}} \! \Bigg) \: \| e(\cdot, T) \|^2_{\Omega} + 
	2 \Big \| \sqrt{\sigma}\, e \Big \|^2_{S_R} \! =: \\
	[ e ]^2_{({\nu},\, {\theta},\, {\zeta}, \,  {\chi})} \leq 
	\majtwod := 
	{\epsilon} \| w(x, T)\|^2_{\Omega} + 
	2L(v, w) + l(v, w) + \\
	\quad \Int\limits_0^T \! \! \left ( \!
	{\rho}_1 \bigg\|\frac{\R_{1,\, \mu} (v, y ,w)}{\sqrt{\lambda}} \bigg\|^2_{\Omega} \! \! + 
	{\rho}_2 {R}^2_{\mathrm{II}, 1}(t) \! + 
	{\alpha}_1(t) {R}^2_{\mathrm{II}, 2}(t)\! + 
	{\alpha}_2(t) \| \R_2 (v, y ,w)\|^2_{A^{-1}} 
	\right ) \! \! \dt,
\end{multline*}
where ${\delta} \in (0, 2]$, 
${\rho}_1 \geq \dfrac{1}{2 - \frac{1}{{\rho}_2}} $, 
${\epsilon} \geq 1$, and 
${\mu} \in [0, 1]$, 
$\R_{1,\, \mu} (v, y ,w)$ and $\R_2 (v, y ,w)$ are defined by (\ref{eq:r-1-mu}) and 
(\ref{eq:r-2}), respectively, and
\begin{equation*}
{R}_{\mathrm{II}, 1} (t) := 
\sqrt{\Sum_{i = 1}^{N} \frac{|\Omega_i|}{\lambda_i} \left(\, \Mean{ \R_{1, 1 - \mu} \, }{\Omega_i} \,\right)^2 }, \;
{R}_{\mathrm{II}, 2} (t) := 
\sqrt{ \Sum_{i = 1}^{N}  \frac{\CPi^2}{\nu_1} \left\| \R_{1, 1 - \mu}  \right\|^2_{\Omega_i}} \,.
\end{equation*}
Here, $\lambda_i = \min\limits_{x \in \Omega_i} \lambda(x, t)$ for a.a. $t \in [0, T]$, 
and 
${\nu} = 2 - {\delta}$,
${\theta} = \sqrt{\left(2 - \frac{1}{{\rho_1}} - \frac{1}{{\rho_2}} \right) \lambda}$, 
$\zeta = 1 - \frac{1}{{\epsilon}}$,
${\chi} = 2$, and 
${\alpha}_1(t)$, ${\alpha}_2(t)$ are 
positive functions satisfying the relation 
$\frac{1}{\alpha_1(t)} + \frac{1}{\alpha_2(t)} = \delta$.
\end{theorem} 

\vskip 15pt
\noindent
For problems, in which $\lambda$ can attain small or zero values we deduce another 
estimate. Assume that
\begin{equation}
\Mean{ \R_{1, 1 - \mu} (v, y, w) }{\Omega_i} = 0, \quad i = 1, ..., N, \quad
\mbox{for a.a.}\; t \in [0, T].
\label{eq:mean-condition-maj-II}
\end{equation}
\begin{theorem}
\label{th:theorem-decomposed-domain-majorant-II-2}
(i) If (\ref{eq:mean-condition-maj-II}) holds,
then for $v , w \in \Ho{1}(Q_T)$ and $y \in Y_{\dvrg}^*(Q_T)$
\begin{multline*}
	(2 - {\delta})\NormA{\nabla e} + 
	\Bigg( 2 - \frac{1}{{\gamma}} \Bigg) \: \Big \| \sqrt{\lambda} e \Big \|^2_{Q_T} +  
	\Bigg( 1 - \frac{1}{{\epsilon}} \Bigg) \: \| e(\cdot, T) \|^2_{\Omega} + 
	2 \Big \| \sqrt{\sigma}\, e \Big \|^2_{S_R} =: \\
	[ e ]^2_{({\nu},\, {\theta},\, {\zeta}, \,  {\chi})} \leq 
	\majtwod  := 
	{\epsilon} \| w(x, T)\|^2_{\Omega} + 
	2L(v, w) + l(v, w) + \qquad \qquad \\
	\Int\limits_0^T \Bigg (
	\gamma \bigg\|\frac{\R_{1,\, \mu} (v, y, w)}{\sqrt{\lambda}} \bigg\|^2_{\Omega} + 
	{\alpha}_1(t) 
	{R}^2_{\mathrm{II}}(t) + 
	{\alpha}_2(t) \| \R_2 (v, y, w) \|^2_{A^{-1}} 
	\Bigg ) \dt,	  
\end{multline*}
where ${\delta} \in (0, 2]$, ${\gamma} \geq \frac12$, ${\epsilon} \geq 1$, and 
${\mu} \in [0, 1]$, 
$\R_{1,\, \mu} (v, y ,w)$ and $\R_2 (v, y ,w)$ are defined by (\ref{eq:r-1-mu}) and 
(\ref{eq:r-2}), respectively, and
\begin{equation}
{R}_{\mathrm{II}}(t) := \sqrt{ \Sum_{i = 1}^{N} \frac{\CPi^2}{\nu_1}  \left\|  \R_{1, 1 - \mu}  \right\|^2_{\Omega_i} \,},
\end{equation}
${\nu} = 2 - {\delta}$,
${\theta} = \sqrt{ \left( 2 - \frac{1}{{\gamma}} \right) \lambda}$, 
${\zeta} = 1 - \frac{1}{{\epsilon}}$,
${\chi} = 2$, and ${\alpha}_1(t)$, ${\alpha}_2(t)$ are 
positive functions satisfying the relation  
$\frac{1}{\alpha_1(t)} + \frac{1}{\alpha_2(t)} = \delta$.
\end{theorem} 


Theorems \ref{th:theorem-decomposed-domain-majorant-II-1} and 
\ref{th:theorem-decomposed-domain-majorant-II-2} can be proven by combining arguments used 
in Theorems \ref{th:theorem-majorant-for-decomposed-domain-1} and 
\ref{th:theorem-majorant-for-decomposed-domain-2}. Since proofs do not contain 
principally new items, we omit these details.

\subsection{Equivalence of 
\texorpdfstring{$\boldsymbol{[ e ]^2_{({\nu},\, {\theta},\, {\zeta})}}$} 
\qquad \: and \texorpdfstring{$\boldsymbol{\majtwo}$} \qquad \:}

\label{ssc:equivalence-of-error-and-advanced-majorant}
We aim to show that the advanced form of the majorant does not lead to an uncontrollable
overestimation of the actual value of the norm 
(\ref{eq:energy-norm-for-reaction-diff-evolutionary-problem}). 
For this purpose, we estimate $\majtwo$ from above and show that this upper bound is 
equivalent to the error norm. Henceforth, we assume that $S_T = S_D$, 
$\beta = {\rm const}$ and $\mu = 0$. As before, these assumption are introduced for the 
sake of simplicity only. Similar estimates for the problems with mixed boundary conditions and 
variable coefficients can be deduced by arguments close to those presented below.

Assume that $y = A \nabla u \in Y^*_{\dvrg}(Q_T)$ and $w = u - v = e$, then 
\begin{alignat}{2}
	\R_1 (v, A \nabla u, e) & = f - {(v + e)}_t - \lambda (v - e) + \dvrg (A \nabla u) = 
	2 \lambda e, \nonumber  \\
	\R_2 (v, A \nabla u, e) & = A \nabla u - A \nabla{(v - e)} = 
	2 A \nabla e.
	\label{eq:r-1-r-2-changed}
\end{alignat}
The functional (\ref{eq:l-function}) can be represented as follows:
\begin{multline}
L(v, e) = \!\! \Int_{Q_T} \!\! \Big( \! v_t \,e + A \nabla v \cdot \nabla e + \lambda v \, e - f e \! \Big) \! \dxt  = \\
\Int_{Q_T} \Big( u_t e + A \nabla u \cdot \nabla e + \lambda u e  - f e \Big) \dxt 
- \Int_{Q_T} \big(  A \nabla e \cdot \nabla e + e_t e + \lambda e^2 \big) \dxt.
\label{eq:l-function-1}	    
\end{multline}
In view of (\ref{eq:parabolic-problem-equation}), the first term in the right-hand side of 
(\ref{eq:l-function-1}) vanishes, and we find that 
\begin{equation}    
L(v, e) = - \Int_{Q_T} \big(  A \nabla e \cdot \nabla e + e_t e + \lambda e^2 \big) \dxt.
\end{equation}    
Next, 
\begin{equation}
l(v, e) = \Int_\Omega \left( |v(x, 0) - \varphi(x)|^2 - 2 e(x, 0) \big( \varphi(x) - v(0, x)\big) \right)\dx = 
- \| e (x, 0)\|_{\Omega}^2.
\label{eq:l-small-function-1}	    
\end{equation}
Let $\frac{4(\beta + 1)}{\delta} = \wp$, then by means of (\ref{eq:e-et-relation}) and
(\ref{eq:l-small-function-1}), we obtain
the estimate
\begin{multline*}
	\majtwo \leq
	\left ( \wp - 2 \right) \NormA{\nabla e} + 
	\left ( \frac{\wp}{\beta} - 2 \right) \| \sqrt{\lambda} e\|^2_{Q_T} + 
	\epsilon \| e(\cdot, T)\|^2_{\Omega} - 2 \Int\limits_{Q_T} e_t e \dxt \leq \\
	\left ( \wp - 2 \right) \NormA{\nabla e} + 
	\left ( \frac{\wp}{\beta} - 2 \right) \| \sqrt{\lambda} e\|^2_{Q_T} + 
	(\epsilon - 1) \| e(\cdot, T)\|^2_{\Omega}.
\end{multline*}
By setting $\hat{\delta} = 2- \delta$, we have
\begin{multline*}
	\majtwo \leq
		\frac{2 \hat{\delta}}{\delta} \left (1 + \frac{2 \beta}{ \hat{\delta}}\right) \NormA{\nabla e} + 
		\frac{2 \hat{\delta}}{\delta} \left (1 + \frac{2}{\hat{\delta}\beta} \right) \| \sqrt{\lambda} e\|^2_{Q_T} + 
		(\epsilon - 1) \| e(\cdot, T)\|^2_{\Omega}.
\end{multline*}
Therefore, for any $v \in \Ho{1} (Q_T)$ we arrive at two-sided estimates
\begin{multline}
[ e ]^2_{(\hat{\nu},\, \hat{\theta},\, \hat{\zeta})} := \hat\delta \NormA{\nabla e} + \hat\gamma \Big \| \sqrt{\lambda} e \Big \|^2_{Q_T} + \hat\epsilon \: \| e(\cdot, T) \|^2_{\Omega} \leq 
\majtwo \leq \\
[ e ]^2_{(\tilde{\nu},\, \tilde{\theta},\, \tilde{\zeta})} :=
\tilde\delta \NormA{\nabla e} + \tilde\gamma \|\sqrt{\lambda} e\|^2_{Q_T} + \tilde\epsilon \| e(\cdot, T)\|^2_{\Omega} \leq 
\mathfrak{C} \, [ e ]^2_{(\hat{\nu},\, \hat{\theta},\, \hat{\zeta})},
\label{eq:majorant-2-estimate}
\end{multline}
%
where 
\begin{multline*}
\hat\gamma = 2, \quad 
\hat\epsilon = \dfrac{\epsilon - 1}{\epsilon} = \dfrac{\tilde\epsilon}{\epsilon}, \quad
\tilde\delta = \frac{2 \hat{\delta}}{\delta} \left( 1 +  \frac{2}{\hat\delta} \right), \quad 
\tilde\gamma =  \frac{2 \hat{\delta}}{\delta} \left (  1 + \frac{2}{\beta \hat{\delta}} \right), \quad 
\tilde\epsilon = \epsilon - 1,
\end{multline*}
and
\begin{equation*}
\mathfrak{C} = 
\max \left\{ 
\frac{2}{\delta} \left( \! 1 +  \frac{2}{\hat\delta} \right), \,
\frac{\hat{\delta}}{\delta} \left ( \! 1 + \frac{2}{\beta \hat{\delta}} \right), \,
\epsilon 
\right\}.
\end{equation*}
The  relation (\ref{eq:majorant-2-estimate}) shows that the quantity 
$\majtwo$ is equivalent to the energy type measure of the error.
This means that the advanced majorant reliably controls deviations from $u$ in
terms of the norm (\ref{eq:energy-norm-for-reaction-diff-evolutionary-problem}).


\section{A lower bound of the deviation from $\boldsymbol{u}$ }
\label{sec:minorant}

Computable minorants of the deviations from exact solutions of partial differential 
equations provide useful information, which allows us to judge on the quantity of the 
error majorants. For elliptic problems having an variational formulation, the minorant 
of the errors can be derived fairly easily by means of the variational arguments 
(see \cite{RepinDeGruyter2008, NeittaanmakiRepin2004}). Below, we derive minorants for 
the considered class of evolutionary problem with the help of a different technique. 


\begin{theorem}
\label{th:theorem-mininum-of-minorant}
Let $v, \: \eta \in \Ho{1}(Q_T)$, then, the following estimate holds:
\begin{multline}
	\Min (\eta, v; \kappa_i): = 
	\sup\limits_{\eta \in \Ho{1}} \Bigg \{  
	\Sum_{i = 1}^{5} G_{v, i}(\eta) + F_{fg\varphi}(\eta) \Bigg \} \leq 
        \error_{(\underline{\nu},\, \underline{\theta},\, \underline{\zeta}, \, \underline{\chi})} : = \\
	\frac{\kappa_1}{2} \NormA{\, \nabla e } + \Bigg \|  \sqrt{ \frac{\kappa_2 + \kappa_3 \lambda}{2}} \, e \, \Bigg \|^2_{Q_T} + 
	\frac{\kappa_4}{2} \| \, e(x, T) \|^2_{\Omega} + \frac{\kappa_5}{2} \| \sqrt{\sigma} e \|^2_{S_R},
	\label{eq:lower-estimate}
\end{multline}
where
\begin{alignat}{2}
G_{v, 1}(\nabla \eta) & = \Int_{Q_T} \Big( - \nabla \eta \cdot A \nabla v - \frac{1}{2 \kappa_1} |\nabla \eta|^2  \Big) \dxt, \quad \nonumber \\
G_{v, 2}(\eta_t) & = \Int_{Q_T} \Big( \eta_t v - \frac{1}{2 \kappa_2}|\eta_t|^2 \Big) \dxt, \quad \nonumber \\
G_{v, 3}(\eta) & = \Int_{Q_T} \lambda \Big( -  v \eta - \frac{1}{2 \kappa_3}|\eta|^2 \Big) \dxt, \quad \quad \nonumber \\
G_{v, 4}\big(\eta(x, T)\big) & = \bigintsss\limits_{\Omega} \bigg( - v(x, T) \eta(x, T) - \frac{1}{2 \kappa_4}|\eta(x, T)|^2 \bigg) \dx, \quad \nonumber \\
G_{v, 5}\big(\eta(s, t)\big) & = \Int_{S_R} \sigma \bigg( - v \eta - \frac{1}{2 \kappa_5 }|\eta|^2 \bigg) \dst, \quad 
\end{alignat}
and
\begin{alignat}{2}
F_{fg \varphi}(\eta) = \Int_{Q_T} f \eta \dxt + \Int_{S_R} g \eta \dst + \Int_{\Omega} \varphi(x) \eta(x, 0) \dx, 
\end{alignat}
where $\underline{\nu} = \frac{\kappa_1}{2}$, 
$\underline{\theta} = \sqrt{ \frac{\kappa_2 + \kappa_3 \lambda}{2} }$, 
$\underline{\zeta} = \frac{\kappa_4}{2}$, 
$\underline{\chi} = \frac{\kappa_5}{2}$, and 
$\kappa_1$, $\kappa_2$, $\kappa_3$, $\kappa_4$, $\kappa_5 > 0$.
\end{theorem}

\begin{proof}
It is not difficult to see that       
\begin{multline}
	\sup\limits_{\eta \in \Ho{1}(Q_T)} 
	\Bigg \{
	\Int_{Q_T} \bigg(
	\nabla \eta \cdot A \nabla e - \frac{1}{2 \kappa_1} |\nabla \eta|^2 - 
        \eta_t e - \frac{1}{2 \kappa_2} |\eta_t|^2 + 
	\lambda \Big( e \eta - \frac{1}{2 \kappa_3} |\eta|^2 \Big) 
	\bigg) \dxt + \\
	\qquad \qquad \qquad \qquad
	\Int_{\Omega} \Big( e(x, T) \eta (x, T) - \frac{1}{2 \kappa_4} |\eta (x, T)|^2 \Big)\dx + 
	\Int_{S_R} \sigma \Big( e \eta  - \frac{1}{2 \kappa_5} |\eta|^2 \Big)\dst \Bigg \} \leq \nonumber \\
	\sup\limits_{\eta \in \Ho{1}(Q_T)} \Int_{Q_T} \bigg( \nabla \eta \cdot A \nabla e - \frac{1}{2 \kappa_1} |\nabla  \eta|^2 \bigg) \dxt + 
	\sup\limits_{\eta_t \in \Ho{1}(Q_T)} \Int_{Q_T} \Big( - \eta_t e - \frac{1}{2 \kappa_2} |\eta_t|^2 \Big) \dxt + \hfill\\
	\sup\limits_{\eta \in \Ho{1}(Q_T)} \Int_{Q_T} \! \! \lambda \Big( e \eta - \frac{1}{2 \kappa_3} |\eta|^2 \Big) \dxt + 
	\sup\limits_{\eta (x, T) \in \Ho{1}(\Omega)} \Int_{\Omega} \! \! \Big( e(x, T) \eta (x, T) - \frac{1}{2 \kappa_4} |\eta (x, T)|^2 \Big)\dx + \\
	\sup\limits_{\eta \in \H{\frac12}(S_R)} \Int_{S_R} \sigma \Big( e \eta  - \frac{1}{2 \kappa_5} |\eta|^2 \Big)\dst.
	\label{eq:quadratic-func-inequality}
\end{multline}
Since
\begin{alignat}{2}
\sup\limits_{\eta \in \Ho{1}(Q_T)}  \Bigg \{ \Int_{Q_T} \bigg( \nabla \eta \cdot A \nabla (u - v) - \frac{1}{2 \kappa_1} |\nabla \eta|^2 \bigg) \dxt  \Bigg \} & \leq 
\frac{\kappa_1}{2} \NormA{\nabla e } \:, \nonumber \\
\sup\limits_{\eta_t \in \Ho{1}(Q_T)}  \Bigg \{ \Int_{Q_T} \Big( - \eta_t e - \frac{1}{2 \kappa_2} |\eta_t|^2 \Big) \dxt \Bigg \} & \leq \frac{\kappa_2}{2} \| e \|^2_{Q_T}, \nonumber \\
\sup\limits_{\eta \in \Ho{1}(Q_T)}  \Bigg \{ \Int_{Q_T} \lambda \Big( e \eta - \frac{1}{2 \kappa_3} |\eta|^2 \Big) \dxt \Bigg \} & \leq \frac{\kappa_3}{2} \| \sqrt{\lambda} e \|^2_{Q_T}, \nonumber \\
\sup\limits_{\eta (x, T) \in \Ho{1}(\Omega)} \Bigg \{ \Int_{\Omega} \Big( e(x, T) \eta (x, T) - \frac{1}{2 \kappa_4} |\eta (x, T)|^2 \Big)\dx \Bigg \} & \leq \frac{\kappa_4}{2} \| e(x, T) \|^2_{\Omega}, \nonumber \\
\sup\limits_{\eta \in \H{\frac12}(S_R)} \Bigg \{ \Int_{S_R} \sigma \Big( e \eta  - \frac{1}{2 \kappa_5} |\eta|^2 \Big) \dst \Bigg \} & \leq \frac{\kappa_5}{2} \| \sqrt{\sigma} e \|^2_{S_R},
\end{alignat}
we find that from one hand 
\begin{multline}
	\sup\limits_{\eta \in \Ho{1}(Q_T)} 
	\Bigg \{
	\Int_{Q_T} \bigg(
	\nabla \eta \cdot A \nabla e - \frac{1}{2 \kappa_1} |\nabla \eta|^2 -
        \eta_t e - \frac{1}{2 \kappa_2} |\eta_t|^2  + \lambda \Big( e \eta - \frac{1}{2 \kappa_3} |\eta|^2 \Big) 
	\bigg) \dxt + \\
	\Int_{\Omega} \Big( e(x, T) \eta (x, T) - \frac{1}{2 \kappa_4} |\eta (x, T)|^2 \Big)\dx + 
	\Int_{S_R} \sigma \Big( e \eta  - \frac{1}{2 \kappa_5} |\eta|^2 \Big)\dst
	\Bigg \} \leq 
        \error_{(\underline{\nu},\, \underline{\theta},\, \underline{\zeta}, \, \underline{\chi})} : = \\
	\frac{\kappa_1}{2} \, \NormA{\nabla e } + \, \Bigg \| \,  \sqrt{ \frac{\kappa_2 + \kappa_3 \lambda}{2}}  \, e  \, \Bigg  \|^2_{Q_T} + 
	\frac{\kappa_4}{2} \, \| \, e(x, T) \, \|^2_{\Omega} + \frac{\kappa_4}{2} \| e(x, t) \|^2_{S_R} .
	\label{eq:quadratic-func-inequality-1}
\end{multline}
%
%
From another hand, (by using (\ref{eq:generalized-statement})) we see that for any $\eta$
the functional 
\begin{multline}
	\! \! \sup\limits_{\eta \in \Ho{1}(Q_T)} 
	\Bigg \{
	\Int_{Q_T} \! \!
  \bigg(
	\nabla \eta \cdot A \nabla e - \frac{1}{2 \kappa_1} |\nabla \eta|^2 
	- \eta_t e - \frac{1}{2 \kappa_2} |\eta_t|^2 + 
	\lambda \Big( e \eta - \frac{1}{2 \kappa_3} |\eta|^2 \Big) \bigg) \dxt + \hfill\\
	\IntO \Big( e(x, T) \eta (x, T) - \frac{1}{2 \kappa_4} |\eta (x, T)|^2 \Big)\dx + 
	\Int_{S_R} \sigma \Big( e \eta  - \frac{1}{2 \kappa_5} |\eta|^2 \Big)\dst \Bigg \} = \\
	\sup\limits_{\eta \in \Ho{1}(Q_T)} \Bigg \{  
	\Sum_{i = 1}^{5} G_{v, i} + F_{fg\,\varphi}(\eta) \Bigg \}
\end{multline}
generates the lower bound of the error norm defined in the right-hand side of the inequality
(\ref{eq:quadratic-func-inequality-1}).

\end{proof}


\end{document}